\documentclass{amsart}

\newtheorem{theorem}{Theorem}[section]
\newtheorem{lemma}[theorem]{Lemma}
\newtheorem*{mainthm}{Main Theorem}

\theoremstyle{definition}
\newtheorem{definition}[theorem]{Definition}

\theoremstyle{remark}
\newtheorem{remark}[theorem]{Remark}

\numberwithin{equation}{section}

\usepackage[top=1.25in, bottom=1.25in, left=1.25in, right=1.25in]{geometry}

\usepackage[utf8]{inputenc}
\usepackage[T1]{fontenc}
\usepackage{graphicx}
\usepackage{mathrsfs}

\usepackage{tikz}
\usepackage{amssymb}
\usepackage{bbm}

\usepackage{amsmath,amssymb,amstext,amsthm}

\DeclareMathOperator{\Lip}{Lip}
\DeclareMathOperator{\dist}{dist}

\DeclareMathOperator{\mix}{mix}

\def\Xint#1{\mathchoice
	{\XXint\displaystyle\textstyle{#1}}%
	{\XXint\textstyle\scriptstyle{#1}}%
	{\XXint\scriptstyle\scriptscriptstyle{#1}}%
	{\XXint\scriptscriptstyle\scriptscriptstyle{#1}}%
	\!\int}
\def\XXint#1#2#3{{\setbox0=\hbox{$#1{#2#3}{\int}$ }
		\vcenter{\hbox{$#2#3$ }}\kern-.6\wd0}}

\def\dashint{\Xint-}

\begin{document}

\title{Cellular Mixing with bounded Palenstrophy}

\author{Gianluca Crippa}
\address{Gianluca Crippa\\
	Department Mathematik und Informatik, Universität Basel}
\curraddr{Spiegelgasse 1, CH-4051 Basel, Switzerland}
\email{gianluca.crippa@unibas.ch}
\thanks{}

\author{Christian Schulze}
\address{Christian Schulze\\
	Department Mathematik und Informatik, Universität Basel}
\curraddr{Spiegelgasse 1, CH-4051 Basel, Switzerland}
\email{christian.schulze@unibas.ch}



\begin{abstract}
	We study the problem of optimal mixing of a passive scalar $\rho$ advected by an incompressible flow on the two dimensional unit square. The scalar $\rho$ solves the continuity equation with a divergence-free velocity field $u$ with uniform-in-time bounds on the homogeneous Sobolev semi-norm $\dot{W}^{s,p}$, where $s>1$ and $1< p \leq \infty$. We measure the \textit{degree of mixedness} of the tracer~$\rho$ via the two different notions of mixing scale commonly used in this setting, namely the \textit{functional } and the \textit{geometric }mixing scale. For velocity fields with the above constraint, it is known that the decay of both mixing scales cannot be faster than \textit{exponential}. Numerical simulations suggest that this exponential lower bound is in fact sharp, but so far there is no explicit analytical example which matches this result. We analyze velocity fields of $\textit{cellular type}$, which is a special localized structure often used in constructions of explicit analytical examples of mixing flows and can be viewed as a generalization of the self-similar construction by Alberti, Crippa and Mazzucato \cite{AlbCrippa}. We show that for any velocity field of cellular type both mixing scales cannot decay faster than \textit{polynomially}.

\end{abstract}

\maketitle
\tableofcontents

	\section{Introduction}

	We consider a passive scalar $\rho$ advected by a time-dependent, divergence-free velocity field $u$ on the two dimensional open unit square $\mathcal{Q}=(-\frac{1}{2},\frac{1}{2})^2$, with $u=0$ on $\partial \mathcal{Q}$.
	Given an initial condition $\bar{\rho}$, the scalar $\rho$ satisfies the Cauchy problem for the continuity equation with velocity field $u$:
	\begin{equation}
		\label{Cauchyprob}
		\begin{cases} \partial_t\rho + \textnormal{div} (u\rho) =0 & \mbox{on } \mathbb{R}_{+}\times\mathbb{R}^2\\ \rho(0,\cdot)=\bar{\rho} & \mbox{on }\mathbb{R}^2.\end{cases}
	\end{equation}
	We assume that $\bar{\rho}$ satisfies $\int_{\mathcal{Q}}\bar{\rho}=0$, thus the condition will be satisfied by the solution $\rho(t,\cdot)$ for all times $t$. Outside of $\mathcal{Q}$, both the velocity field $u$ and the solution $\rho$ are identically zero for all times. The results presented in this paper hold in all dimensions $d\geq 2$ as well with the necessary changes in the scaling analysis. For the construction of examples, $d=2$ is the most restrictive case. \\\\
	We study the problem of \textit{optimal mixing} in this setting. The first issue one has to address is how to quantify the \textit{degree of mixedness} of the tracer. There are mainly two notions of mixing scale which are commonly used in this setting: the \textit{functional mixing scale}, which measures the semi-norm of the tracer $\rho$ in the homogeneous Sobolev space $\dot{H}^{-1}$ (see \cite{Multi}, \cite{LinThif} and \cite{AlbCrippa} for the definition of homogeneous Sobolev spaces), and the \textit{geometric mixing scale} $\mathcal{G}$, first introduced by A. Bressan for \textit{binary} solutions $\rho\in\lbrace-1,+1\rbrace$ (see \cite{Bressan}) in connection with a conjecture on the cost of rearrangements of sets. This second scale determines the smallest possible radius $r$ such that in each ball of radius $r$ the proportion of both level sets is comparable. We will use a canonical generalization of this scale for any bounded solution, not necessarily binary (see \cite{AlbCrippa}).  Although these two notions of mixing scales are somewhat related, one can verify with elementary examples (see \cite{Lin}) that they are not equivalent.\\\\
Studying mixing mechanisms of a quantity transported by a flow is complex and of interest in a wide range of applications such as food processing, oceanography and atmospheric science. Furthermore, the efficiency of many chemical reactions is heavily influenced by the degree of mixedness of the reactants (for example air and fuel in a combustion motor).  \\\\
	The theory of \textit{optimal mixing} addresses the following question. Let an initial condition $\bar{\rho}$, a mixing scale $\mix(\cdot)$ (geometric or functional, for instance) and a physical constraint of interest on the velocity field $u$ (for example a uniform-in-time bound on either the \textit{kinetic energy}, the \textit{enstrophy} or the \textit{palenstrophy}) be given. What is the fastest possible decay of the function $\mix(\rho(t,\cdot))$ achievable, and can we construct explicit examples which achieve the maximal rate of decay (such velocity fields would be called \textit{optimal mixers})?
	\\\\
	To summarize one particular result of this type, let us consider a uniform-in-time bound on the~$\dot{H}^1$ norm of the velocity field $u$ (a so called \textit{enstrophy} constraint) or more generally, a uniform-in-time bound on the $\dot{W}^{1,p}$ norm, where $1<p\leq\infty$. Crippa and De Lellis \cite{LellisCrippa} showed that under this constraint the decay of the geometric mixing scale $\mathcal{G}$ cannot be faster than \textit{exponential}, i.e.
	\begin{equation}
	\label{cnn}
		\mathcal{G}(\rho(t,\cdot))\geq C e^{-ct}
	\end{equation}
	where $C>0$ and $c>0$ are constants depending on the initial datum $\bar{\rho}$ and on the given bounds on the velocity field (see also L\'eger \cite{leger}).  Iyer, Kiselev and Xu \cite{Kiselev}, as well as Seis \cite{Seis}, later showed that an \textit{exponential} lower bound for the \textit{functional} mixing scale holds as well, i.e.
	\begin{equation}
	\label{norreaga}
		\|\rho(t,\cdot)\|_{\dot{H}^{-1}}\geq C e^{-ct}.
	\end{equation}
	These lower bounds turn out to be \textit{optimal}. For any $p<\frac{3+\sqrt{5}}{2}$,  Yao and Zlato\v{s} \cite{Yao} constructed explicit examples of solutions with a uniform in time bound on $\|\nabla u(t,\cdot)\|_{p}$ and a solution $\rho$ whose mixing scales decay at an exponential rate. Using an ansatz of self-similarity, Alberti, Crippa and Mazzucato \cite{AlbCrippaCras}, \cite{AlbCrippa} later constructed for all $1\leq p\leq\infty$ examples with a uniform in time bound on $\|\nabla u(t,\cdot)\|_{p}$ and a solution $\rho$, whose mixing scales decay at an exponential rate.\\\\
	We address the case where $u$ has a uniform-in-time bound on the $\dot{H}^2$ norm (fixed \textit{palenstrophy} constraint), or more generally, a uniform-in-time bound on the $\dot{W}^{s,p}$ norm, where $s>1$ and $1< p\leq\infty$ (to which we will also refer with a slight abuse of terminology as a ``fixed palenstrophy'' constraint). In this case, there is currently no proof of a decay estimate which is also optimal, in contrast with the case described above.\\\\
 Combining the exponential lower bounds under fixed enstrophy \eqref{cnn} and \eqref{norreaga} with the (fractional) Poincar\'e inequality 
	\begin{equation*}
		\|u(t,\cdot)\|_{\dot{W}^{s,p}(\mathbb{R}^2)}\geq C_{s,p} \|u(t,\cdot)\|_{\dot{W}^{1,p}(\mathbb{R}^2)}
	\end{equation*}
	which holds for all velocity fields $u$ as in our setting, we immediately inherit an exponential lower bound under fixed palenstrophy as well. Whether or not this lower bound is sharp, however, is not entirely clear. At the present time, the fastest decay reached by an explicit (analytical) example under fixed palenstrophy is (only) \textit{polynomial}:
	\begin{equation}
	\label{generaleasy}
		\|\rho(t,\cdot)\|_{\dot{H}^{-1}}\simeq Ct^{-\frac{1}{s-1}}\, .
	\end{equation}
	For this construction, see \cite{AlbCrippaCras}, \cite{AlbCrippa} (quasi self-similar mixing). \\\\
	However, Lunasin \textit{et al.} \cite{Lin} exhibited a \textit{numerical} example of a velocity field which appears to mix at an exponential rate under fixed palenstrophy, which suggests that the exponential lower bound is indeed sharp also under this constraint. This triggers the search for analytical examples matching this result. 
	\\\\
	Constructing explicit analytical examples of mixing flows is a  difficult task, and it is also a fairly recent trend. Until now there are only a handful of examples available, whatever the constraint on the velocity field (\cite{Bressan},\cite{Depauw},\cite{LinThif},\cite{Lin},\cite{Yao},\cite{AlbCrippa}). Most of those examples are using the same basic structure for the construction, a basic structure we will later formalize and refer to as \textit{``of cellular type''}. Roughly speaking, the idea is the following.\\\\
	We start with an initial condition $\bar{\rho}$ with zero average on the $2$-dimensional square $\mathcal{Q}$. In a first step we subdivide $\mathcal{Q}$ into $4$ disjoint sub-squares $D_1,\ldots,D_4$ of equal size. We now try to construct a velocity field $u_0$ in such a way that the solution $\rho(1,\cdot)$ has zero average on each of the sub-squares, meaning that the tracer gets equally distributed among $D_1,\ldots,D_4$, as schematically visualized in~Figure \ref{faschtfertig}. 
		\begin{figure}[h]
			\label{faschtfertig}
			\begin{center}
				\scalebox{0.35}{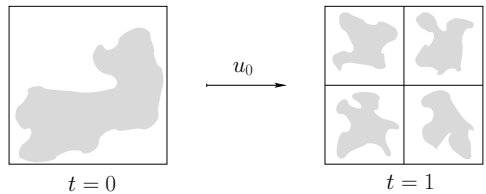}
			\end{center}
				\caption{Example of a cellular basic building block}
		\end{figure}
		
	We continue inductively by subdividing each $D_i$ into four equal sub-squares and distributing the tracer equally among those sub-squares with \textit{tracer movements localized in $D_i$} respectively, and so forth. In a sense, in each step the problem of mixing on a large area is transferred into the problem of \textit{separately} mixing on four smaller areas. \\\\
	In our definition of a flow of cellular type (see Definition \ref{celltype}) we make one further assumption. For reasons explained in more detail in Remark \ref{platel}, we add a condition which keeps a mixer from mixing \emph{too well} before entering the next iteration. Without such a condition, one would be allowed to mix arbitrarily well on unit scale before restricting the tracer movements on smaller scales. We note that the examples mentioned above fulfill this additional condition as well. \\\\
	Using a cellular type structure for the construction of mixing flows holds many benefits. In addition to the convenient inductive nature of the construction, a crucial benefit of this structure is that it allows to keep track of both the geometric and the functional mixing scale at all times (see~Lemma \ref{tilingmixing}). Also, we note that under a fixed enstrophy constraint, the fastest possible (exponential) decay was reached by examples with a velocity field of cellular type (\cite{Yao}, \cite{AlbCrippa}). This means that at least under this constraint, the ansatz of a cellular structure does not a priori slow down the mixing process in a significant way. \\\\
	Our main objective in this paper is to show that, in the case of a palenstrophy contraint, this ansatz actually has the effect to slow down the mixing to a polynomial rate. In detail: \\\\ 
	\textbf{Main Result.} With a velocity field of cellular type no decay faster than polynomial can be achieved under a fixed palenstrophy constrain on the velocity field. More precisely, for any evolution of cellular type with a velocity field $u\in L_t^\infty(\dot{W}^{s,p})$, where $s>1$, the corresponding solution $\rho$ satisfies
		\begin{equation}
		\mathcal{G}(\rho(T_n,\cdot))\geq C T_n^{-\frac{1}{s-1}}\hspace{0.5cm}\text{ and }\hspace{0.5cm}\|\rho(T_n,\cdot)\|_{\dot{H}^{-1}(\mathbb{R}^2)}\geq CT_n^{-\frac{2}{s-1}}\, 
		\end{equation}
		for time steps $T_n$ that are specified later.\\\\
		Going back to the example mentioned before \eqref{generaleasy}, our result explains the polynomial decay is not due to the specific ``geometrically rigid'' assumption of self-similarity, but rather to a more general constraint of increased localization of the particles.\\\\
The main strategy for the construction of the velocity field in the numerical simulation of Lunasin~\textit{et~al.}~\cite{Lin} is to instantaneously maximize the depletion of the $H^{-1}$ norm. The resulting flow appears to do large scale movements throughout the simulation and nowhere the flow appears to be localized (this can be seen in Figure 4 in \cite{Lin}). The simulation is therefore very compatible with our result, since one of the defining characteristics of any velocity field of cellular type is the increased localization of its flow. Both results therefore suggest that in order to construct a velocity field which mixes at an exponential rate under fixed palenstrophy one needs to take a velocity field which moves the passive scalar at unit scale at all times, for instance making out of the numerical example in~\cite{Lin} an analytical example. This appears to be a very challenging task. An analytical example of a velocity field which is not of cellular type and mixes a large class of initial data (although with a polynomial rate) is presented in \cite{EulerMix}. The velocity field in such example is in fact time-independent and for all times it induces a large scale particle motion. \\\\
As a side result we also show that velocity fields of cellular type cannot be \emph{universal mixers}, meaning that a fixed velocity field of cellular type cannot mix \emph{every} initial condition (whatever the rate of mixing).\\\\
	The rest of the paper is organized as follows. In Chapter 2 we state the definitions of the two mixing scales. Before stating the precise definition of a cellular type structure in Chapter \ref{cellmix}, we briefly summarize the key ideas of the construction based on an ansatz of self-similarity \cite{AlbCrippa} in Chapter \ref{selfsimaiu}. In Chapter \ref{mainresuaiu} we state the Main Theorem and describe the strategy for its proof. The proof consists of two main steps: a minimal cost estimate which we show in Chapter \ref{MinCostEachStep}, and a scaling argument which we implement in Chapter \ref{ScalingAnalysis}. A few lemmas for the previous chapters are proved in Appendix A and B. In Appendix \ref{brent} we show that velocity fields of cellular type cannot be universal mixers.   \\\\
	\textbf{Acknowledgements.} This work was partially supported by the ERC Starting Grant 676675 FLIRT. We thank one of the anonymous referees for several detailed remarks that led to a shortening of the proof of Theorem~\ref{yolinski} and as a consequence to a simplification of the notion of characteristic length scale in Definition~\ref{zulu}.
 
	\section{Definitions of mixing scales and tiling} 
	\label{defmixtile}
	We begin by introducing the two notions of mixing scales for solutions of \eqref{Cauchyprob} on the two-dimensional square $\mathcal{Q}$. The first definition is a canonical adaption of the geometric mixing scale of~\cite{Bressan}, see Definition 2.11 in \cite{AlbCrippa}:
	\begin{definition}[Geometric Mixing Scale]
		\label{dpgc}
		 Given an accuracy parameter $0<\kappa<1$, the geometric mixing scale of $\rho(t,\cdot)$ is the infimum $\epsilon(t)$ of all $\epsilon>0$ such that for every $x\in \mathbb{R}^2$ there holds 
		\begin{equation}
			\label{AggroBerlin}
			\frac{1}{\|\rho(t,\cdot)\|_{\infty}}\left|\,\dashint_{B_\epsilon(x)}\rho(t,y)\,dy\right|\leq \kappa\, .
		\end{equation}
We denote by $\mathcal{G}(\rho(t,\cdot))$ such infimum $\epsilon(t)$. The $\kappa$ in the definition plays a minor role and is fixed for the rest of this paper.
	\end{definition}
	\noindent
	The second mixing scale is the \textit{functional}	mixing scale:
	\begin{definition}[Functional Mixing Scale]
		The functional mixing scale of $\rho(t,\cdot)$ is $\|\rho(t,\cdot)\|_{\dot{H}^{-1}(\mathbb{R}^2)}$.
	\end{definition}
	\noindent
	Next, we formalize the notion of subdividing $\mathcal{Q}$ into a family of disjoint sub-squares:
	\begin{definition}[Tiling on $\mathcal{Q}$] Let a tiling parameter $\lambda>0$, such that $\lambda^{-1}$ is an integer greater or equal two, be given. We denote by $\mathcal{T}_\lambda$ the \textit{tiling} of $\mathcal{Q}$ with squares of side $\lambda$, consisting of the $1/\lambda^2$ open squares in $\mathcal{Q}$ of the form 
		\begin{equation*}
			\left\lbrace (x,y)\in \mathcal{Q}\, :\, -\frac{1}{2}+k\lambda<x<-\frac{1}{2}+(k+1)\lambda \textnormal{ and }-\frac{1}{2}+h\lambda<y<-\frac{1}{2}+(h+1)\lambda\right \rbrace
		\end{equation*}
		for $k,h=0,\ldots,1/\lambda -1$.
			\begin{figure}[h]
				\begin{center}
					\scalebox{.4}{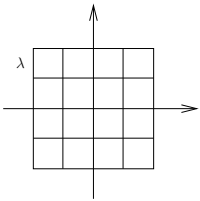}
				\end{center}
				\caption{Tiling with tiling parameter $\lambda$}
			\end{figure}
	\end{definition}
	The next lemma states that if the tracer $\rho$ is equally distributed  among all sub-squares $Q\in\mathcal{T}_{\lambda}$, then both the geometric and the functional mixing scales of $\rho$ are of order $\lambda$ or better:
	\begin{lemma}
		\label{tilingmixing}
		Let $\rho$ be a bounded function such that 
		\begin{equation}
			\label{dunix}
			\dashint_Q \rho\,dy=0
		\end{equation}
		for every tile $Q\in \mathcal{T}_{\lambda}$. Then there exist constants $C_1=C_1(\kappa)$ and $C_2=C_2(\|\rho\|_\infty)$ such that
		\begin{itemize}
			\item[(i)]
			
			$\mathcal{G}(\rho)\leq C_1\lambda$,
			
			\item[(ii)]
			$\|\rho\|_{\dot{H}^{-1}}\leq C_2 \lambda$.
		\end{itemize}
	\end{lemma}
	\begin{proof}
		For readability, we moved the proof to the Appendix \ref{L1Proof}
	\end{proof}

	\section{Self-similar mixing}
	\label{selfsimaiu}
	In order to properly motivate our definition of cellular type structure, we briefly summarize the key points of the construction by Alberti, Crippa and Mazzucato in \cite{AlbCrippa} using an ansatz of self-similarity.\\\\
	We fix a tiling parameter $\lambda>0$ and a time parameter $\tau>0$ and define the time steps
	\begin{equation}
		\label{timesteps}
		T_n=\sum_{i=0}^{n-1}\tau^i\hspace{1cm}\textnormal{ for } n=1,2,\ldots,\infty\, .
	\end{equation}
The key element in the construction is a velocity field $u_0$ with corresponding solution $\rho_0$ on the two dimensional square $\mathcal{Q}$ defined in the time interval $0\leq t\leq 1$, such that
	\begin{itemize}
		\item [(i)] $u_0$ is bounded in $\dot{W}^{s,p}(\mathbb{R}^2)$ uniformly in time and divergence free and tangent to $\partial\mathcal{Q}$ ;
		\item [(ii)] For each $Q\in\mathcal{T}_{\lambda}$ we have that
		\begin{equation*}
			\rho_0(1,x)=\rho_0\left(0,\frac{x-r_Q}{\lambda}\right)\hspace{0.5cm}\text{for each }x\in Q\, ,
		\end{equation*} 
		where $r_Q$ is the center of the square $Q$.
	\end{itemize}
\begin{figure}[h]
	\label{selfsimfig}
		\begin{center}
			\scalebox{0.35}{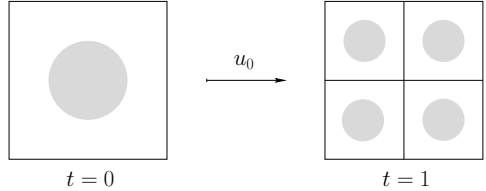}
		\end{center}
		  \caption{Example of a self-similar basic building block}
			\label{selfsimilarel}
\end{figure}
	This pair of velocity field $u_0$ with corresponding solution $\rho_0$ serves as a basis, or a \textit{basic building block} for the remainder of the construction of the velocity field $u$ with corresponding solution $\rho$, in the sense that in each time interval $[T_n,T_{n+1})$, we write $u$ and $\rho$ as the sum of rescaled translations of $u_0$ and $\rho_0$, respectively, as we are going to describe now. \\\\
	For times $T_n\leq t< T_{n+1}$, in any tile $Q\in \mathcal{T}_{\lambda^n}$ we set $u$ and $\rho$ as
	\begin{equation}
		\label{tennisweezy}
		u(t,x)=\frac{\lambda^n}{\tau^n}u_0\left(\frac{t-T_n}{\tau^n},\frac{x-r_Q}{\lambda^n}\right)\hspace{0.5cm}\textnormal{and}\hspace{0.5cm}\rho(x,t)=\rho_0\left(\frac{t-T_n}{\tau^n},\frac{x-r_Q}{\lambda^n}\right),
	\end{equation}
	where $r_Q$ is the center of the cube $Q$. Note that, since $\rho_0(0,\cdot)$ is mean-free, conditions (ii) and \eqref{tennisweezy} imply that
	\begin{equation}
		\label{ohmy}
		\dashint_{Q}\rho(t,y)\,dy=0
	\end{equation}
	for all $Q\in\mathcal{T}_{\lambda^n}$ and all $T_n\leq t< T_{n+1}$, which by Lemma \ref{tilingmixing} ensures that $\mix(\rho(t,\cdot))\sim\lambda^n$ on $[T_n,T_{n+1})$ (or better). In a last step, a scaling argument shows that
	\begin{itemize}
		\item[(i)] If $u$ is bounded in $\dot{W}^{1,p}(\mathbb{R}^2)$ uniformly in time (enstrophy constraint), we can set the time parameter $\tau=1$, which yields an \textit{exponential} decay: 
		\begin{equation*}
			\mathcal{G}(\rho(T_n,\cdot))\simeq C\lambda^{T_n}\hspace{0.5cm}\text{ and }\hspace{0.5cm}\|\rho(T_n,\cdot)\|_{\dot{H}^{-1}(\mathbb{R}^2)}\simeq C\lambda^{\frac{T_n}{2}}\,;
		\end{equation*}
		\item[(ii)] If $u$ is bounded in $\dot{W}^{s,p}(\mathbb{R}^2)$, where $s>1$ uniformly in time (palenstrophy constraint), it is necessary to set the time parameter $\tau>1$, which yields the following \textit{polynomial }decay:
		\begin{equation*}
		\mathcal{G}(\rho(T_n,\cdot))\simeq C T_n^{-\frac{1}{s-1}}\hspace{0.5cm}\text{ and }\hspace{0.5cm}\|\rho(T_n,\cdot)\|_{\dot{H}^{-1}(\mathbb{R}^2)}\simeq CT_n^{-\frac{2}{s-1}}\, .
		\end{equation*}	
	\end{itemize}
	\begin{remark}
		To be more precise, it is not evident that a basic element as in Figure \ref{selfsimilarel} with a uniform-in-time bound on the $H^2$-norm exists. For this reason, the construction is done using \textit{quasi} self-similar basic elements (for details see \cite{AlbCrippa}, Chapter 6). This is however only a technical difference: the basic idea of the construction remains the same.
	\end{remark}
	
	\section{Cellular mixing}
	\label{cellmix}
	In this chapter we introduce our notion of cellular flow. As in the self-similar construction, given a time parameter $\tau>0$ we define the time steps
	\begin{equation*}
		T_n=\sum_{i=0}^{n-1}\tau^i\hspace{1cm}\textnormal{ for } n=1,2,\ldots,\infty
	\end{equation*}
	and fix a tiling parameter $\lambda>0$, such that $\lambda^{-1}$ is an integer greater or equal two. We also fix for all the rest of the paper an accuracy parameter $\kappa\in (0,1)$ for the geometric mixing scale. For this reason we will often not explicitly state dependencies of constants on $\kappa$.\\\\
	To summarize the previous section, we say that a pair of a velocity field $u$ with corresponding solution $\rho$ has a \textit{self-similar} structure if on any time interval $T_n\leq t < T_{n+1}$ both $u$ and $\rho$ can be written as the sum of rescaled translations of \textit{one} basic building block $(u_0,\rho_0)$ as in \eqref{tennisweezy}. We now generalize this structure by simply allowing a very general set $\mathcal{B}$ of basic building blocks $(u_0,\rho_0)$, which keep two key properties of the basic building block used in the self-similar construction. The idea is to replace the precise assumption on the geometry of the basic building block in the self-similar case with a more general assumption on the mixing scale only. In particular, this means that~$\mathcal{B}$ can be an infinite family. We then say that a pair of a velocity field $u$ with corresponding solution $\rho$ to the continuity equation \eqref{Cauchyprob} is \textit{of cellular type} if it can be patched together with basic elements $(u_0,\rho_0)\in\mathcal{B}$. See Definition \ref{basicbuildingdefin} for the precise definition.\\\\
	The first key property we adapt from the self-similar basic building block is the equal distribution of the tracer among all tiles, i.e. for all $(u_0,\rho_0) \in\mathcal{B}$ we have that
	\begin{equation}
		\label{kasalla}
		\int_{Q}\rho_0(1,x)\, dx=0
	\end{equation}
	for all $Q\in\mathcal{T}_{\lambda}$. \\\\
The second key property we implement is a condition which keeps a basic building block from mixing \textit{too well}. For this, we introduce the \textit{characteristic length scale} of a set, which is inspired by a similar length scale introduced in \cite{Kiselev} (proof of Lemma 2.6).
	We fix a parameter $\bar{s}\in(0,1)$. This parameter will be fixed for the rest of this paper and for this reason explicit dependencies of constants on $\bar{s}$ will often not be stated. 
	\begin{definition} [Characteristic length scale]
		\label{zulu}
		Let a set $A\subset E$ of positive measure be given. We denote by $\mathcal{C}=\mathcal{C}(E, A)$ the set of all admissible balls $B(x,r) \subset E$ such that
		\begin{equation}
		\label{saucinonu}
		\frac{|A\cap B(x,r)|}{|B(x,r)|}>1-\frac{1-\kappa}{2}\cdot\bar{s}.
		\end{equation}
		We define the \textit{characteristic length scale} of the set $A$ with respect to $E$ as
		\begin{equation*}
		LS_E(A)=\sup\left\lbrace r \,|\,B(x,r)\in \mathcal{C}\right\rbrace.
		\end{equation*}
	\end{definition}
	\begin{remark}
This scale determines the largest radius of a ball which violates \emph{significantly} (that is, with the uniform gap $1 - \bar s$)
condition \eqref{AggroBerlin} of the geometric mixing scale for any binary tracer $\rho$ and will serve as our measurement of how un-mixed the tracer is. As an example, consider a solution $\rho$ of the following type:
			\begin{equation*}
			\rho(t,\cdot)|_{\mathcal{Q}}=\begin{cases}1 & \mbox{on }A\\ c & \mbox{else , }\end{cases}
			\end{equation*}
		where $0<|A|\leq \frac{1}{2}$ and the constant $c$ is chosen such that $\rho$ is mean-free (in particular $-1\leq c<0$). Let us determine how small $\epsilon$ has to be, in order that a ball $B(x,r)$ such that
				\begin{equation}
				\frac{|A\cap B(x,r)|}{|B(x,r)|}>1-\epsilon
				\end{equation}
				violates condition \eqref{AggroBerlin} of the geometric mixing scale. We compute
				\begin{equation}
				\label{frappedyrap}
				\begin{split}
				\left|\,\dashint_{B(x,r)}\rho(t,y)\,dy\right|&=\frac{1}{|B(x,r)|}\left(1\cdot|A\cap B(x,r)|+c\cdot \left|B(x,r)\setminus A\right|\right)\\
				&>1-2\epsilon\stackrel{!}{\geq}\kappa
				\end{split}
				\end{equation}
				and hence any $\epsilon\leq\frac{1-\kappa}{2}$ works.  In particular, any ball in \eqref{saucinonu} violates condition \eqref{AggroBerlin} of the geometric mixing scale. Notice that the notion of characteristic length scale becomes weaker as the parameter $\bar s$ approaches the value $1$. 
			\end{remark}
				In summary, we define a basic building block $(u_0,\rho_0)$ as follows.
				\begin{definition}[Cellular basic building block]
					\label{basicbuildingdefin}
					Let a tiling parameter $\lambda>0$, real values $1<p\leq \infty$, $s>1$, a set size $0<\theta\leq \frac{1}{2}$ and a parameter $a>0$ be given. We define $\mathcal{B}=\mathcal{B}_{\lambda,a,p,\theta,s}$ as the set of all pairs $(u_0,\rho_0)$ such that:
					\begin{itemize}
						\item [(i)] $u_0\in L_t^{\infty}(\dot{W}_x^{s,p})$ on the time interval $0\leq t\leq 1$ and is divergence free, $u_0=0$ on $\partial \mathcal{Q}$ and zero outside of $\mathcal{Q}$;
						\item [(ii)] $\rho_0$ is a solution to the Cauchy problem \eqref{Cauchyprob} on the time interval $0\leq t\leq 1$ with velocity field $u_0$ and initial value
						\begin{equation}
						\label{woreport}
						\rho_0(0,\cdot)=\begin{cases}1 & \mbox{on }A\\ c & \mbox{on }\mathcal{Q}\setminus A\end{cases}
						\end{equation} 
						where $LS_{\mathcal{Q}}(A)\geq a$ and $|A|=\theta$;
						\item[(iii)] For every $Q\in \mathcal{T}_\lambda$ we have that
						\begin{equation}
						\label{equaldistri}
						\int_{Q}\rho_0(1,x)\, dx=0.
						\end{equation}	
					\end{itemize}
				\end{definition}
				\begin{definition}[Solution of cellular type]
					\label{celltype}
					We say that a pair of a velocity field $u$ with corresponding solution $\rho$ to the continuity equation $\eqref{Cauchyprob}$ is \textit{of $(\lambda,a,p,\theta,s)$-cellular type}, or simply \textit{of cellular type} when parameters are fixed, if on any time interval $T_n\leq t < T_{n+1}$ and for any tile $Q\in \mathcal{T}_{\lambda^n}$ there exists a $(u_0,\rho_0)\in\mathcal{B}_{\lambda,a,p,\theta,s}$ (depending on $Q$), such that
					\begin{equation}
					\label{jeffwecan}
					u(t,x)=\frac{\lambda^n}{\tau^n}u_0\left(\frac{t-T_n}{\tau^n},\frac{x-r_Q}{\lambda^n}\right)\hspace{0.5cm}\textnormal{and}\hspace{0.5cm}\rho(t,x)=\rho_0\left(\frac{t-T_n}{\tau^n},\frac{x-r_Q}{\lambda^n}\right)\hspace{0.25cm}\textnormal{ for }x\in Q
					\end{equation}
					where $r_Q$ is the center of the tile $Q$.
				\end{definition}
	
	\begin{remark}
		\label{miauu}
		Note that as a consequence of \eqref{woreport}, \eqref{equaldistri} and \eqref{jeffwecan}, for any $Q\in\mathcal{T}_{\lambda^n}$ we have that
	\begin{equation*}
	\rho(T_n,\cdot)|_Q=\begin{cases}1 & \mbox{on }B\\ c & \mbox{else, }\end{cases}
	\end{equation*}
where $B=B(Q,n)$ is a set such that $|B|=\lambda^{2n}\theta$ and $LS_{Q}(B)\geq \lambda^n a$.
\end{remark}
	\begin{remark}
				\label{platel}
The idea in the definition of a cellular flow is to reach a quality improvement of the mixedness of the passive tracer by equally redistributing the tracer in each step among a finer sub-tiling. By Lemma \ref{tilingmixing}, we expect the mixing scales to decay in each step by a factor $\lambda$ on average due \textit{solely} to this effect. Without implementing any condition which keeps a basic building block from mixing \textit{too well}, however, a basic building block would be allowed to mix to arbitrarily small scales, with tracer movements constrained only to $\mathcal{Q}$, as schematically illustrated in Figure \ref{koda}.
\begin{figure}[h]
\begin{center}
	\scalebox{0.35}{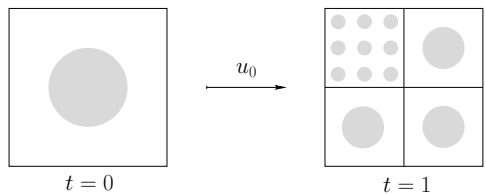}
\end{center}
\caption{Example of a basic building block mixing too well on a large scale}
	\label{koda}
\end{figure}
For such evolutions it is in general not possible to derive lower bounds on the mixing scale better than exponential. In order to exclude this possibility we fix the parameter $a$ in Definition~\ref{celltype}. This forces a cellular basic building block to be in a certain sense \textit{``self-similar with respect to the degree of mixedness''}, meaning the following. While for a self-similar basic building block $(u_0,\rho_0)$ the solution $\rho_0(1,\cdot)$ is the sum of rescaled versions of $\rho_0(0,\cdot)$, for a cellular basic buidling block $(u_0,\rho_0)$ we have that $\rho_0(1,\cdot)$ can be written as the sum of rescaled versions of $\tilde{\rho}_1,\ldots,\tilde{\rho}_{\lambda^{-2n}}$, where each $\tilde{\rho}_i$ is unmixed to the same degree as $\rho_0(0,\cdot)$.
	\end{remark}
	
	\begin{remark}
		In our case, the set $E$ in Definition \ref{zulu} will always be either $\mathcal{Q}$ or a tile $Q\in\mathcal{T}_{\lambda}$, which are both types of sets for which $LS_E(A)\in\mathbb{R}$ for any measurable set $A\subset E$ of positive measure, since for a.e. point $x \in A$ there exists $r>0$ such that $B(x,r)\subset E$ and
		\begin{equation*}
			\lim_{r\to 0}\frac{|A\cap B(x,r)|}{|B(x,r)|}=1
		\end{equation*}
		and therefore $\mathcal{C}$ is non-empty.
	\end{remark}
	The following lemma states that there are lower bounds on the mixing scale for cellular flows.
	\begin{lemma}
		\label{fireworks}
	Let $\rho$ be a solution of $(\lambda,a,p,\theta,s)$-cellular type. Then there exists a constant $C$ such that
			\begin{itemize}
				\item [(i)] $\mathcal{G}(\rho(T_n,\cdot))\geq C \lambda^n a$,
				\item [(ii)] $\|\rho(T_n,\cdot)\|_{\dot{H}^{-1}(\mathbb{R}^2)}\geq C \lambda^{2n} a$.
			\end{itemize}
	\end{lemma}
	\begin{proof}
 For readability, we moved the proof to the Appendix \ref{birdman}
	\end{proof}

	\section{Main result and sketch of the proof}
	\label{mainresuaiu}
	For a parameter $\tau>0$ we define the time steps
	\begin{equation*}
		T_n=\sum_{i=0}^{n-1}\tau^i\hspace{1cm}\textnormal{ for } n=1,2,\ldots,\infty\, .
	\end{equation*}
	Let further a tiling parameter $\lambda>0$, a set size $\theta>0$, and parameters $a>0$ and $p>1$ be given.
	\begin{mainthm}
		Let a divergence-free velocity field $u\in L_t^\infty(\dot{W}^{s,p})$ where $s>1$ and a solution $\rho$ to the continuity equation \eqref{Cauchyprob}  of $(\lambda,a,p,\theta,s)$-cellular type be given. Then the decay of both the geometric and the functional mixing scale cannot be faster than polynomial, that is, for all $n\in\mathbb{N}$ we have that
			\begin{equation}
			\label{topdawg}
			\mathcal{G}(\rho(T_n,\cdot))\geq C T_n^{-\frac{1}{s-1}}\hspace{0.5cm}\text{ and }\hspace{0.5cm}\|\rho(T_n,\cdot)\|_{\dot{H}^{-1}(\mathbb{R}^2)}\geq CT_n^{-\frac{2}{s-1}}
			\end{equation}	
			where $C=C(a,\lambda,s)$.
	\end{mainthm}

	\begin{remark}[for all $T_n$ vs. for all $t\geq 0$]
 Note that condition \eqref{topdawg} a priori only holds for all time steps $T_n$, where $n\in\mathbb{N}$. This is sufficient in order to show that the decay of both mixing scales cannot be faster than polynomial and therefore to exclude the possibility of exponential mixing by a cellular flow. We could however also get a polynomial lower bound which holds for all $t\geq 0$, i.e.
		\begin{equation}
		\label{cigar}
		\mathcal{G}(\rho(t,\cdot))\geq C t^{-\frac{1}{s-1}}\hspace{0.5cm}\text{ and }\hspace{0.5cm}\|\rho(t,\cdot)\|_{\dot{H}^{-1}}\geq Ct^{-\frac{2}{s-1}}
		\end{equation}
		for all $t\geq 0$, by preventing a basic building block  $(u_0,\rho_0)$ from mixing \textit{too well} also in the interior of the time interval $(0,1)$. This can be achieved for example by adding the following constraint: 
		\begin{itemize}
		\item[(iv)] For all $t\in(0,1)$ we have that
		\begin{equation*}
		\rho_0(t,\cdot)|_{\mathcal{Q}}=\begin{cases}1 & \mbox{on }B_{t}\\ c & \mbox{else }\end{cases}
		\end{equation*}
		where $LS_{\mathcal{Q}}(B_{t})\geq \lambda a$.
	\end{itemize}
	\end{remark}
\begin{remark}[sharpness of the polynomial lower bound]
The quasi-self similar construction in \cite{AlbCrippa} is a special case of a cellular flow which achieves polynomial decay under a constraint in $L^{\infty}(\dot{W}^{s,p})$ (where $s>1$ and $p>1$) and therefore the polynomial lower bound in the Main Theorem is \textit{sharp}. For completeness, in the Appendix \ref{grischun} we show that any cellular flow which is patched together with basic building blocks with certain uniform bounds will achieve \textit{at least} a polynomial decay. This is a simple generalization of the quasi self-similar computations in Chapter 6 of \cite{AlbCrippa}.
\end{remark}

	\begin{proof}[Proof of the Main Theorem]
	\emph{Step 1: Proof for fine tilings.}	We will first show the result for a sufficiently fine tiling, i.e. assuming
		\begin{equation}
			\label{lambdachoice}
			0<\lambda \leq C(a,\bar s)
		\end{equation}
		where $C(a)$ is a constant which will be determined in Theorem \ref{yolinski}.
		For such a tiling, the proof consists of two main parts:
		\begin{itemize}
			\item [(i)]
			The gradient of any $u_0$, where $(u_0,\rho_0)\in\mathcal{B}_{\lambda,a,p,\theta,s}$, satisfies the following lower bound:
			\begin{equation}
				\label{MinCostner}
				M_{\lambda,a,p}:=\inf_{(u_0,\rho_0)\in\mathcal{B}}\left\lbrace \int\limits_{0}^1\|\nabla u_0(t,\cdot)\|_{L^p}\,dt\right\rbrace > 0\,. 
			\end{equation}
			The details of this estimate will be in Section \ref{MinCostEachStep}.
			\item[(ii)]
			Using estimate \eqref{MinCostner}, a scaling argument in Section \ref{ScalingAnalysis} will show that if $\|u(t,\cdot)\|_{\dot{W}^{s,p}}$ has to be uniformly bounded in time, it is necessary to choose the time parameter $\tau$ \textit{strictly greater than }$1$. This will be sufficient to conclude that the decay of the mixing scale is as in \eqref{topdawg}.
		\end{itemize}
		This concludes the proof for all $0<\lambda<C(a,\bar s)$. \\
		
		\emph{Step 2: Reducing the general case to Step 1.} Choosing a fine tiling as in \eqref{lambdachoice} is something we can do without loss of generality. Indeed, in the case where $\frac{1}{2}\geq\lambda\geq C(a,\bar s)$, we can linearly rescale the pair $u$ and $\rho$ in time in such a way that the rescaled pair $(\tilde{u}, \tilde{\rho})$ takes the form of a cellular flow with time parameter $\tilde{\tau}=\tau^l$ and tiling parameter $\tilde{\lambda}=\lambda^l$, where $0<\lambda^l\leq C(a,\bar s)$. Since the rescaled pair fulfills \eqref{lambdachoice}, using the first part of the proof we deduce that the decay of $\mix(\tilde{\rho}(t,\cdot))$ is at most polynomial, which implies that also the decay of $\mix(\rho(t,\cdot))$ is at most polynomial (since it is a linear rescaling). The details of this argument are given in Appendix \ref{Finetiling}.
	\end{proof}
	\section{Minimal cost in each step}
	\label{MinCostEachStep}
	In this chapter we prove the estimate \eqref{MinCostner} in \emph{Step 1} part (i) in the proof of the Main Theorem. The minimal cost estimate for a sufficiently fine tiling presented in Theorem \ref{yolinski} is mainly an adaptation of Lemma 2.6 in \cite{Kiselev} for our setting. The key element in the proof is the following regularity result for regular Lagrangian flows by Crippa and De Lellis \cite{LellisCrippa} (for the result and definition of regular Lagrangian flows, see \cite{LellisCrippa}):
	\begin{theorem}
		[Crippa and De Lellis]
		\label{CrippaDeLellis}
		Let $\Phi(x)=\Psi(1,x)$ be the flow map of the (incompressible) vector field $u$, let $p>1$ be given. For every $\eta>0$, there exists a set $E\subset \mathcal{Q}$ and a constant $C=C(p)$, such that $|E|\leq \eta$ and  
		\begin{equation}
			\label{Chuia}
			\Lip(\Phi^{-1}|_{E^c})\leq \exp\left(\frac{C}{\eta^{1/p}}\int_{0}^{1}\|\nabla u(s,\cdot)\|_{L^p}\,ds\right)\, .
		\end{equation}
		Here
		\begin{equation*}
		\Lip(\Phi^{-1}|_{E^c}):= \sup_{\substack{x,y\in E^c \\ x\neq y}}\frac{\left| \Phi^{-1}(x)-\Phi^{-1}(y)\right|}{|x-y|}\, .
		\end{equation*}
	\end{theorem}
	\begin{theorem}[Minimal Cost]
		\label{yolinski}
		Let $\rho$ be a solution to the continuity equation with velocity field $u$ such that
		\begin{equation*}
			\rho_0(0,\cdot)=\begin{cases}1 & \mbox{on }A\\ c & \mbox{else , }\end{cases}
		\end{equation*}
		where $LS_\mathcal{Q}(A)\geq a$ and $|A|=\theta\leq \frac{1}{2}$. Then there exists a tiling parameter $\lambda=\lambda(a)$ such that if at time $t=1$ we have that
		\begin{equation}
			\label{conditione}
			\int_Q\rho(1,y)\,dy=0
		\end{equation}
		for any $Q\in\mathcal{T}_{\lambda}$, then 
		\begin{equation}
		\label{lerequire}
			\int\limits_{0}^1\|\nabla u(t,\cdot)\|_{L^p}\,dt\geq C(a,p,\bar s)>0\,,
		\end{equation}
		where the constant $C=C(a,p)$ does not depend on $u$.
	\end{theorem}
	
	\begin{proof}
		In order to understand the basic geometric idea of the proof, for simplicity let us assume that  there exists a ball $B=B(x,r)$ with radius $r\geq\frac{3}{4} a$ which is \textit{entirely} filled with points of the set $A$. Let us further assume $\Phi^{-1}$ is Lipschitz continuous on $\mathcal{Q}$, with the Lipschitz constant fulfilling
		\begin{equation}
			\label{wrongLip}
			\Lip(\Phi^{-1})\leq \exp\left(\int_{0}^{1}\|\nabla u(s)\|_{L^p}\,ds\right)
		\end{equation}
		neglecting the dependence on the set $E$ in \eqref{Chuia}. 
		\begin{figure}[h]
		\begin{center}
			\label{yolodude}
\scalebox{.4}{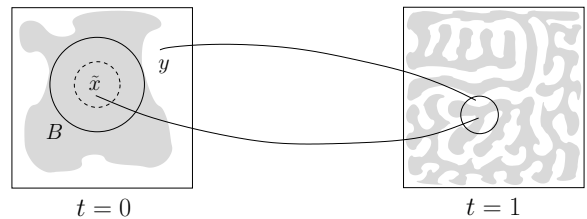}
		\end{center}
		\caption{Geometric idea for the proof of Theorem~\ref{yolinski}}
		\end{figure}
		The idea is to take any point $\tilde{x}\in B(x,\frac{r}{2})\cap A$. Due to condition \eqref{conditione}, we know that wherever $\Phi(\tilde{x})$ lands, there exist many points $\Phi(y)$, such that $|\Phi(y)-\Phi(\tilde{x})|\leq \sqrt{2} \lambda$ and $y\in A^c$. Since the ball $B$ is filled entirely with points of the set $A$, this means in particular that $y\in B^c$ and therefore $|y-\tilde{x}|\geq \frac{r}{2}$. Hence, choosing $\lambda\leq \frac{3}{16\sqrt{2}}a$, we have that
		\begin{equation}
			\label{loltofeelgood}
			\begin{split}
			2=\frac{\frac{3}{8}a}{\frac{3}{16}a}\leq\frac{|\tilde{x}-y|}{|\Phi(\tilde{x})-\Phi(y)|}&\leq \sup_{x\neq y}\frac{|\Phi^{-1}(x)-\Phi^{-1}(y)|}{|x-y|}\\
			&=:\Lip\left(\Phi^{-1}\right)\leq \exp\left(\int_{0}^{1}\|\nabla u(s,\cdot)\|_{L^p}\,ds\right)
			\end{split}
		\end{equation}
		and therefore
		\begin{equation*}
			\int_{0}^{1}\|\nabla u(s,\cdot)\|_{L^p}\,ds\geq \log(2)>0.
		\end{equation*}
		This concludes the geometric idea of the proof. \\\\The difference in our setting is that, since $LS_\mathcal{Q}(A)\geq a$, there exists a ball $B$ with radius $\frac{3}{4} a$ which is filled only to a large proportion with points of the set $A$. The other difference is that the Lipschitz constant depends the size of the set $E$ in Theorem \ref{CrippaDeLellis}. The rest of proof is a matter of playing with the constants in order to show that points $\tilde{x}$ and $y$ exist, just as above.\\\\
			Since $A$ is a set of characteristic length scale $a$, by definition there exists a ball $B=B(x,r)\subset\mathcal{Q}$, such that $r\geq\frac{3}{4}a$ and 
			\begin{equation}
			\label{fbfli}
			\frac{|A\cap B|}{|B|}>1-\frac{1-\kappa}{2}\cdot\bar{s}\geq 1-\frac{\bar{s}}{2}.
			\end{equation}	
			We define a ball $\tilde{B}=B(x,\sigma)\subset B$ with the same center as $B$, such that $|\tilde{B}|=(\frac{3+\bar{s}}{4})|B|$. Thus
			\begin{equation}
			|A\cap\tilde{B}|\geq |\tilde{B}|-\frac{\bar{s}}{2}|B|\geq \left(\frac{3+\bar{s}}{4}-\frac{\bar{s}}{2}\right)\pi r^2=\left(\frac{3-\bar{s}}{4}\right)\pi r^2 .
			\end{equation}
			Let $\left\lbrace Q_i\right\rbrace_{i=1}^{1/\lambda^2}$ denote the sub-squares in the tiling $\mathcal{T}_{\lambda}$. For $i=1,\ldots,\frac{1}{\lambda^2}$, denote
			\begin{equation}
			A_i:=\Phi\left(A\cap\tilde{B}\right)\cap Q_i\, .
			\end{equation} 
			Note that the $A_i$'s are disjoint for different $i$, and $\sum_{i=1}^{1/\lambda^2}|A_i|\geq\left(\frac{3-\bar{s}}{4}\right)\pi r^2$.\\
			Let $G_i :=\Phi(A^c)\cap Q_i$. We further decompose each $G_i$ into $G_i^{\text{ good}}$ and $G_i^{\text{ bad}}$, where
			\begin{equation}
			G_i^{\text{ good}}:=\Phi(A^c)\cap Q_i\cap\Phi(B^c) \hspace{0.5cm}\text{ and }\hspace{0.5cm}G_i^{\text{ bad}}:=\Phi(A^c)\cap Q_i\cap\Phi(B).
			\end{equation}
			Recall that \eqref{fbfli} yields that $\sum_{i=1}^{1/\lambda^2}|G_i^{\text{ bad}}|\leq\frac{\bar{s}}{2}\pi r^2 $, which we will use later.\\
			By Theorem \ref{CrippaDeLellis}, for any $\eta>0$ (let us fix $\eta=\frac{1-\bar{s}}{4}\pi r^2$), there exists $E\subset \mathcal{Q}$ with $|E|<\eta$, such that
			\begin{equation}
			\label{ovaso}
			\Lip(\Phi^{-1}|_{E^c})\leq \exp\left(\frac{C}{\eta^{1/p}}\int_{0}^{1}\|\nabla u(s,\cdot)\|_{L^p}\,ds\right)\, .
			\end{equation}
			For such $E$, we claim that there exists some $i\in\lbrace 1,\ldots ,\frac{1}{\lambda^2}\rbrace$, such that \emph{both} $A_i\setminus E$ and $G_i^{\text{ good}}\setminus E$ are nonempty. Once this is proved, taking any $\tilde{x}\in A_i\setminus E$ and $y\in G_i^{\text{ good}}\setminus E$ yields
			\begin{equation}
			\Phi^{-1}(\tilde{x})\in A\cap \tilde{B}\hspace{0.5cm}\text{ and }\hspace{0.5cm}\Phi^{-1}(y)\in A^c\cap B^c
			\end{equation}
			Thus $|\Phi^{-1}(\tilde{x})-\Phi^{-1}(y)|\geq r-\sigma$ (note that $r-\sigma>0$ is a quantity depending on $a$ and $\bar{s}$ only). On the other hand, since $\tilde{x},y\in Q_i$ we have that $|\tilde{x}-y|\leq\sqrt{2}\lambda$. Therefore, as long as $\lambda<\frac{r-\sigma}{2 \sqrt{2}}$ we have that $\Lip(\Phi^{-1}|_{E^c})\geq 2$. Combining this with \eqref{ovaso} concludes the proof.\\\\
			It then remains to prove the claim. For each $i=1,\ldots,\frac{1}{\lambda^2}$, note that
			\begin{equation}
			\begin{split}
			\min\lbrace|A_i\setminus E|,\, |G_i^{\text{ good}}\setminus E|\rbrace& \geq \min\lbrace |A_i|, |G_i|\rbrace-|G_i^{\text{ bad}}|-|Q_i\cap E| \\
			& =  |A_i | - |G_i^{\text{ bad}}| - |Q_i \cap E| \,.
			\end{split}
			\end{equation}
			Summing it up for $i=1,\ldots,\frac{1}{\lambda^2}$ yields
			\begin{equation}
			\sum_{i=1}^{1/\lambda^2}\min\lbrace|A_i\setminus E|,\, |G_i^{\text{ good}}\setminus E|\rbrace\geq \left(\frac{3-\bar{s}}{4}\right)\pi r^2-\frac{\bar{s}}{2}\pi r^2-\frac{1-\bar{s}}{4}\pi r^2=\frac{1-\bar{s}}{2}\pi r^2>0
			\end{equation}
			hence the claim is proved.
				\end{proof}

	\section{Scaling analysis}
	\label{ScalingAnalysis}
	In this chapter we implement \emph{Step 1} part (ii) in the proof of the Main Theorem.
\subsection{Scaling analysis for Sobolev spaces of integer order} In this section we implement the scaling analysis computations for the  case where $s=k\geq 2 $ is an integer and $\lambda$ is a sufficiently small tiling parameter such that by Theorem \ref{yolinski} we have that
	\begin{equation}
	M_{\lambda,a,p}=\inf_{(u_0,\rho_0)\in\mathcal{B}}\left\lbrace \int\limits_{0}^1\|\nabla u_0(t,\cdot)\|_{L^p}\,dt\right\rbrace > 0 \, .
	\end{equation} 
	The computations for fractional Sobolev spaces are in Section \ref{ScalingFrac} and the argument for a general tiling parameter is in the Appendix \ref{Finetiling}. \\\\
 
	We compute $\|\nabla_x^k u(t,\cdot)\|_{L^p(\mathcal{Q})}$ on the time interval $[T_n,T_{n+1})$. Remember that for any $Q\in\mathcal{T}_{\lambda^n}$ the vector field $u$ takes on the form
	\begin{equation*}
		\frac{\lambda^n}{\tau^n}u_Q\left(\frac{t-T_n}{\tau^n},\frac{x-r_Q}{\lambda^n}\right)\, ,
	\end{equation*}
	where $(u_Q,\rho_Q)\in \mathcal{B}$. Using change of variables, the Poincar\'e inequality followed by the Jensen inequality we get
	\begin{equation}
		\label{bigcalc}
		\begin{split}
			\int\limits_{T_n}^{T_{n+1}}\|\nabla^k u(t,\cdot) \|_{L^p(\mathcal{Q})}^p\,dt&=\int\limits_{T_n}^{T_{n+1}}\sum_{Q\in\mathcal{T}_{\lambda^n}}\int\limits_{Q}\left| \nabla_x^k\left(\frac{\lambda^n}{\tau^n}u_Q\left(\frac{t-T_n}{\tau^n},\frac{x-r_Q}{\lambda^n}\right)\right)\right|^p\,dx\,dt\\
			&=\frac{1}{\lambda^{(k-1)np}}\frac{\tau^n}{\tau^{np}}\sum_{Q\in\mathcal{T}_{\lambda^n}}\int\limits_{0}^{1}\int\limits_{Q}\left| \left(\nabla_x^k u_Q \right)\left(t,\frac{x-r_Q}{\lambda^n}\right)\right|^p\,dx\,dt\\
			&=\frac{\tau^n}{(\lambda^{k-1}\tau)^{np}}\lambda^{2n}\sum_{Q\in\mathcal{T}_{\lambda^n}}\int\limits_{0}^{1}\int\limits_{\mathcal{Q}}\left| \left(\nabla_x^k u_Q \right)\left(t,y\right)\right|^p\,dx\,dt\,\\
			&\stackrel{P}{\geq}\frac{C_{p,k}\tau^n}{(\lambda^{k-1}\tau)^{np}}\lambda^{2n}\sum_{Q\in\mathcal{T}_{\lambda^n}}\int\limits_{0}^{1}\int\limits_{\mathcal{Q}}\left| \left(\nabla_x u_Q \right)\left(t,y\right)\right|^p\,dx\,dt\,\\
			&\stackrel{J}{\geq} \frac{C_{p,k}\tau^n}{(\lambda^{k-1}\tau)^{np}}M_{\lambda,a,p}^p\, ,
		\end{split}
	\end{equation}
	where $M_{\lambda,a,p}$ is the minimal cost in \eqref{MinCostner}. Now since $\left|[T_n,T_{n+1}]\right|=\tau^n$, together with \eqref{bigcalc} we get that
	\begin{equation}
		\left\|\|\nabla_x^k u(t,\cdot)\|_{L^p(\mathcal{Q})}^p\right\|_{L_t^{\infty}(T_n,T_{n+1})}\geq \left(\frac{1}{\lambda^{k-1} \tau}\right)^{np} \underbrace{C_{p,k} M_{\lambda,a,p}^p }_{>0}
	\end{equation}
	and hence, since $u\in L_t^\infty(\dot{W}^{s,p})$ we must have that $\frac{1}{\lambda^{k-1}\tau}\leq1$, that is
	\begin{equation}
	 1<\frac{1}{\lambda^{k-1}}\leq \tau\, .
	\end{equation}
	Therefore, for $n$ large enough, we have that
	\begin{equation}
		\label{randooom}
		T_n=\frac{\tau^n-1}{\tau-1}\geq\frac{\lambda^{-(k-1)n}-1}{\lambda^{1-k}-1}\, ,
	\end{equation}
	which is equivalent to
	\begin{equation}
	\label{rudimentari}
	\lambda^n\geq  \left(\frac{1}{T_n(\lambda^{1-k}-1)+1}\right)^{\frac{1}{k-1}}\sim T_n^{-\frac{1}{k-1}}\, .
	\end{equation}
	Combining Lemma \ref{fireworks} and \eqref{rudimentari} we therefore get that
		\begin{equation}
		\label{brolinski}
		\mathcal{G}(\rho(T_n,\cdot))\geq C \lambda^n \sim T_n^{-\frac{1}{k-1}}\hspace{0.5cm}\text{ and }\hspace{0.5cm}\|\rho(T_n,\cdot)\|_{\dot{H}^{-1}}\geq C \lambda^{2n}\sim T_n^{-\frac{2}{k-1}}
		\end{equation}
		for all $n\in\mathbb{N}$. This concludes the proof for the case where $s=k$ is an integer.
		
			\subsection{Scaling analysis for fractional Sobolev spaces}
			\label{ScalingFrac}
			For the case $s=k+r$, where $r\in (0,1)$ and $k\geq 1$ is an integer, we use the following definition of the fractional semi-norm (in two dimensions)
			\begin{equation*}
			\|f\|_{\dot{W}^{s,p}(\mathbb{R}^2)}^p=\iint\limits_{\mathbb{R}^2\times\mathbb{R}^2}\frac{|\nabla^k f(x)-\nabla^k f(y)|^p}{|x-y|^{2+rp}}\,dx\,dy\, ,
			\end{equation*}
			as well as the following fractional Poincar\'e inequality:
		\begin{lemma}
			Let $r\in (0,1)$, $1\leq p < \infty$ and $\mathcal{Q}$ the unit cube (in $\mathbb{R}^2$). Then
			\begin{equation*}
			\int_{\mathcal{Q}}|f(x)-f_{\mathcal{Q}}|^p\,dx\leq C(r,p)\iint\limits_{\mathcal{Q}\times\mathcal{Q}}\frac{| f(x)- f(y)|^p}{|x-y|^{2+rp}}\,dx\,dy
			\end{equation*}
			where $f_{\mathcal{Q}}=\dashint_{\mathcal{Q}}f(y)\,dy$ .
		\end{lemma}
		\begin{proof}
			Using the Jensen inequality, we have that
		\begin{align}
		\int_{\mathcal{Q}}|f(x)-f_{\mathcal{Q}}|^p \, dx
		& 
		= \int\limits_{\mathcal{Q}} \left|f(x)- \dashint_{\mathcal{Q}} f(y) dy \right|^p\,dx
		\\ \nonumber
		& \leq \dashint\limits_\mathcal{Q}  \int\limits_\mathcal{Q} | f(x) -  f(y) |^p\, dx dy 
		\\ \nonumber
		& \leq C(r,p) \int\limits_\mathcal{Q}  \int\limits_\mathcal{Q} \frac{| f(x) -  f(y) |^p}{|x-y|^{2+rp}} \, dx dy \, . \qedhere
		\end{align} 
		\end{proof}
			 For the scaling analysis in the fractional case we compute
			\begin{equation*}
			\begin{split}
			\int_{T_n}^{T_{n+1}} \|u(t,\cdot)\|_{\dot{W}^{s,p}(\mathbb{R}^2)}^p\,dt&=\int_{T_n}^{T_{n+1}}\iint\limits_{\mathbb{R}^2\times\mathbb{R}^2}\frac{|\nabla^k u(t,x)-\nabla^k u(t,y)|^p}{|x-y|^{2+rp}}\,dx\,dy\,dt\\
			&\geq \sum\limits_{Q\in\mathcal{T}_{\lambda^{n}}}\int_{T_n}^{T_{n+1}}\iint\limits_{Q\times Q}\frac{|\nabla^k u(t,x)-\nabla^k u(t,y)|^p}{|x-y|^{2+rp}}\,dx\,dy\,dt\\
			&=\frac{\tau^n}{\tau^{np}}\frac{\lambda^{4n}}{\lambda^{(k-1)np}}\sum\limits_{Q\in\mathcal{T}_{\lambda^{n}}}\int_{0}^{1}\iint\limits_{\mathcal{Q}\times\mathcal{Q}}\frac{|\nabla^k u_Q(t,x)-\nabla^k u_Q(t,y)|^p}{\left|\lambda^n(x-y)\right|^{2+rp}}\,dx\,dy\,dt\\
			&\stackrel{P}{\geq} C\frac{\tau^n}{\tau^{np}}\frac{1}{\lambda^{(k-1+r)np}} \lambda^{2n} \sum\limits_{Q\in\mathcal{T}_{\lambda^{n}}}\int_{0}^{1}\int_{\mathcal{Q}}\left|\nabla u_Q(t,y) \right|^p\,dy\,dt\\
			&\stackrel{J}{\geq} C\frac{\tau^n}{\tau^{np}}\frac{1}{\lambda^{(k-1+r)np}}M_{\lambda,a,p}^p\\
			&=C\frac{\tau^n}{\tau^{np}}\frac{1}{\lambda^{(s-1)np}}M_{\lambda,a,p}^p\, .
			\end{split}
			\end{equation*}
			Hence we get that
			\begin{equation*}
			\left\|\|u(t,\cdot)\|_{\dot{W}^{s,p}(\mathcal{Q})}^p\right\|_{L_t^\infty(T_n,T_{n+1})}\geq\left(\frac{1}{\lambda^{(s-1)}\tau}\right)^{np} C M_{\lambda,a,p}^p
			\end{equation*}
			and therefore, if we ask of $\|u(t,\cdot)\|_{\dot{W}^{s,p}(\mathcal{Q})}^p$ to be uniformly bounded in time, we need that
			\begin{equation*}
			1<\left(\frac{1}{\lambda}\right)^{s-1}\leq \tau\, .
			\end{equation*}
		We can therefore conclude as in the integer case, that:
			\begin{equation}
			\label{ggn}
			\mathcal{G}(\rho(T_n,\cdot))\geq C_2 \lambda^n \sim T_n^{-\frac{1}{s-1}}\hspace{0.5cm}\text{ and }\hspace{0.5cm}\|\rho(T_n,\cdot)\|_{\dot{H}^{-1}}\geq \tilde{C}_2 \lambda^{2n}\sim T_n^{-\frac{2}{s-1}}\, .
			\end{equation}

	\appendix
	\section{Proof of Lemmas \ref{tilingmixing} and \ref{fireworks}, Fine tiling argument }
	\label{ExtraStuff}
In this chapter of the Appendix we collect the proofs of the auxiliary Lemmas \ref{tilingmixing} and \ref{fireworks}, as well as the proof of the reduction in \emph{Step 2} in the proof of the Main Theorem.
	\subsection{Proof of Lemma \ref{tilingmixing}}
	\label{L1Proof}
	\begin{proof}
		$(i)$ We extend the tiling $\mathcal{T}_{\lambda}$ on $\mathcal{Q}$ to a tiling $\tilde{\mathcal{T}}_{\lambda}$ on the entire plane $\mathbb{R}^2$ (with tiles of side $\lambda$). Since $\rho$ is identically zero outside of $\mathcal{Q}$, condition \eqref{dunix} is fulfilled by all $Q\in\tilde{\mathcal{T}}_{\lambda}$. The rest of the argument is precisely Lemma 3.5 in \cite{AlbCrippa} with $C_1=\frac{4\sqrt{2}}{\kappa}$.\\\\
		$(ii)$ We use the Poincar\'e inequality
		\begin{equation*}
			\|u-u_Q\|_{L^p(Q)}\leq C\lambda\|\nabla u\|_{L^p(Q)}
		\end{equation*}
		where $Q\in \mathcal{T}_\lambda $, $u_Q=\dashint_Q u$ and $C=C(p)$ as well as the definition of the $\dot{H}^{-1}$ norm via duality
		\begin{equation}
			\label{anotherway}
			\|\rho\|_{\dot{H}^{-1}(\mathbb{R}^2)}=\sup\left\lbrace \int_{\mathbb{R}^2}\rho(x)\xi(x)\,dx \,:\, \|\nabla \xi \|_{L^2(\mathbb{R}^2)}\leq 1 \right\rbrace \, .
		\end{equation}
		Now let $\xi$ such that $\|\nabla \xi\|_{L^2}\leq 1$ be given. Then
		\begin{equation}
			\label{yolin}
			\begin{split}
				\left|\int_{\mathbb{R}^2}\rho(x)\xi(x)\,dx\right|&\leq \sum_{Q\in\mathcal{T}_\lambda}\left|\int_{Q}\rho(x)\xi(x)\,dx \right|\\
				&\leq\sum_{Q\in\mathcal{T}_\lambda}\int_{Q}\left|\rho(x)(\xi(x)-\xi_Q)\right|\,dx +\left|\int_{Q}\rho(x)\xi_Q\,dx\right|\\
				&\leq\|\rho\|_{\infty}\sum_{Q\in\mathcal{T}_\lambda}\|\xi-\xi_Q\|_{L^1(Q)}\\
				&\leq C\|\rho\|_{\infty}\lambda \sum_{Q\in\mathcal{T}_\lambda}\|\nabla \xi\|_{L^1(Q)}\\
				&= C\|\rho\|_{\infty}\lambda \|\nabla \xi\|_{L^1(\mathcal{Q})}\leq C\|\rho\|_{\infty}\lambda \|\nabla \xi\|_{L^2(\mathcal{Q})}\\
				&\leq C\|\rho\|_{\infty}\lambda \, .
			\end{split}
		\end{equation}
		Combining \eqref{yolin} with \eqref{anotherway}, we can choose $C_2=C\|\rho\|_{\infty}$.
	\end{proof}
	\subsection{Proof of Lemma \ref{fireworks} }
	\label{birdman}
	\begin{proof}
					 (i) By the definition of the characteristic length scale, there exists a ball $B(x,r)\subset \mathcal{Q}$ with $r\geq \frac{3}{4}a\lambda^n$ such that
			\begin{equation}
			\frac{|B_{Q,n}\cap B(x,r)|}{|B(x,r)|}>1-\frac{1-\kappa}{2}\cdot \bar{s} \,.
			\end{equation}
		Since $B(x,r)$ violates condition \eqref{AggroBerlin} of the geometric mixing scale, we have that
		\begin{equation*}
\mathcal{G}(\rho(T_n,\cdot))\geq C \lambda^n a\, .
		\end{equation*}
		(ii) We showed in (i) that there exists a ball $B(x,r)\subset \mathcal{Q}$ where $r\geq \frac{3}{4}a\lambda^n$ such that
				\begin{equation}
				\frac{1}{\|\rho\|_{\infty}}\left|\,\dashint_{B(x,r)}\rho(T_n,y)\,dy\right|> \kappa \, .
				\end{equation}
		
			We can define a smooth test function (see Lemma A.1 in \cite{Yao} or Lemma 2.3 in \cite{Kiselev}) $g:\mathbb{R}^2\to [0,1]$ such that
		\begin{equation*}
		g=\begin{cases}1 & \mbox{on }B(x,r)\\ 0 & \mbox{outside }B(x,r (1+\frac{\kappa}{20}))\end{cases}
		\end{equation*}
		such that $\|g\|_{\dot{H}^1}\leq \frac{C_1}{\sqrt{\kappa}}$. Hence
		\begin{equation*}
		\left|\int_{\mathbb{R}^2}\rho(t, y)g(y)\,dy\right|\geq \left|\int_{B(x,r)}\rho(t,y)g(y)\,dy\right|-\frac{3\pi\kappa r^2}{20}\|\rho\|_{\infty}\geq \|\rho\|_{\infty} \pi\kappa r^2\left(1-\frac{3}{20}\right)
		\end{equation*}
		and therefore
		\begin{equation*}
		\|\rho(t,\cdot)\|_{\dot{H}^{-1}}\geq \frac{1}{\|g\|_{\dot{H}^1}} \left|\int_{\mathbb{R}^2}\rho(t,y)g(y)\,dy\right|\geq C(\kappa,\|\rho\|_{\infty}) r^2\geq C(\kappa,\|\rho\|_{\infty}) \lambda^{2n}a\,.
		\end{equation*}
		In our case we set $\|\rho\|_{\infty}=1$, and $\kappa$ is fixed, so $C(\kappa,\|\rho\|_{\infty})=C$.
	\end{proof}

		\subsection{Fine-tiling argument}
		\label{Finetiling}
		Let a divergence-free velocity field $u\in L_t^\infty(\dot{W}^{s,p})$ where $s>1$ with corresponding solution $\rho$ of cellular type be given as in the Main Theorem, with a tiling parameter $\frac{1}{2}\geq\lambda\geq C(a,\bar s)$, a time parameter $\tau>0$ and corresponding time steps
		\begin{equation*}
		T_n=\sum_{i=0}^{n-1}\tau^i\,.
		\end{equation*}
		We first note that there exists an $l\in\mathbb{N}$ such that		
		\begin{equation*}
		0<\tilde{\lambda}:=\lambda^l<C(a,\bar s) \,.
		\end{equation*}	
		\noindent
		\textbf{Claim 1.} Under the above assumptions, we have that $\tau\neq 1$.
		\begin{proof}
			Let us assume that $\tau=1$. Then the rescaled pair
			\begin{equation*}
			\tilde{u}(t,x)=l u\left(lt,x\right)\hspace{0.5cm}\textnormal{and}\hspace{0.5cm}\tilde{\rho}\left(t,x\right)=\rho\left(lt,x\right)
			\end{equation*}
			is of cellular type with $\tau=1$ and tiling parameter $0<\lambda^l<C(a,\bar s)$. Also, clearly $u\in L_t^\infty(\dot{W}^{s,p})$. This is a contradiction to the first part of the proof of the Main Theorem (where we assumed a sufficiently fine tiling).
		\end{proof}
		\noindent
		\textbf{Claim 2.}
		By the above Claim we have that $\tau\neq 1$. The rescaled pair 
		\begin{equation}
		\label{Zulu}
		\tilde{u}(t,x)=C u\left(Ct,x\right)\hspace{0.5cm}\textnormal{and}\hspace{0.5cm}\tilde{\rho}\left(t,x\right)=\rho\left(Ct,x\right),\hspace{0.4cm}\textnormal{where }\hspace{0.4cm}C=\frac{1-\tau^l}{1-\tau}
		\end{equation}
		is also of cellular type with a tiling parameter $\tilde{\lambda}=\lambda^l$ and time steps
		\begin{equation*}
		\tilde{T}_n=\sum_{i=0}^{n-1}\tilde{\tau}^i\, ,
		\end{equation*}
		where $\tilde{\tau}=\tau^l$.
		\begin{proof}
			This follows directly from the identity
			\begin{equation}
			\label{bigleb}
			C\tilde{T}_n=T_{nl}\hspace{0.75cm}\textnormal{ for } n=1,2,\ldots,\infty
			\end{equation}
			which we verify by induction:
			\begin{enumerate}
				\item $n=1$: 
				\begin{equation*}
				C\tilde{T}_1=C\cdot 1=\frac{1-\tau^l}{1-\tau}=T_l
				\end{equation*}
				\item $n\to n+1$:
				\begin{equation*}
				\begin{split}
				C\tilde{T}_{n+1}=C\tilde{T}_{n}+C\tilde{\tau}^{n}=T_{nl}+\frac{1-\tau^l}{1-\tau}\tau^{nl}&=\frac{1-\tau^{nl}+\tau^{nl}-\tau^{(n+1)l}}{1-\tau}\\
				&=\frac{1-\tau^{(n+1)l}}{1-\tau}=T_{(n+1)l}\, .
				\end{split}
				\end{equation*}
			\end{enumerate}
			Using identity \eqref{bigleb}, we verify that for all $n=1,2,\ldots$ we have that
			\begin{equation}
			\int_{Q}\tilde{\rho}(\tilde{T}_n,y)\,dy=\int_{Q}\rho(C\tilde{T}_n,y)\,dy=\int_{Q}\rho(T_{nl},y)\, dy=0
			\end{equation}
			for all $Q\in \mathcal{T}_{\lambda^{ln}}=\mathcal{T}_{{\tilde{\lambda}}^n}$.
			And similarly, for all $Q\in \mathcal{T}_{\tilde{\lambda}^n}=\mathcal{T}_{\lambda^{nl}}$ we have that
			\begin{equation}
			\tilde{\rho}(\tilde{T}_n,\cdot)|_Q=\rho(T_{nl},\cdot)|_Q=\begin{cases}1 & \mbox{on }B\\ -1 & \mbox{else , }\end{cases}
			\end{equation}
			where $B$ is of is a set of characteristic length scale greater or equal to $\lambda^{nl}a=\tilde{\lambda}^{n}a$.
		\end{proof}
		Since $\tilde{u}\in L_t^\infty(\dot{W}^{s,p})$ with corresponding solution $\tilde{\rho}$ is of cellular type with fine enough tiling as in \eqref{lambdachoice}, we can use the scaling result \eqref{brolinski}  and \eqref{ggn} to conclude 
		\begin{equation}
		\mathcal{G}(\tilde{\rho}(\tilde{T}_n,\cdot))\geq  \tilde{T}_n^{-\frac{1}{s-1}}\hspace{0.5cm}\text{ and }\hspace{0.5cm}\|\tilde{\rho}(\tilde{T}_n,\cdot)\|_{\dot{H}^{-1}}\geq  \tilde{T}_n^{-\frac{2}{s-1}}\, .
		\end{equation}
		Therefore by definition \eqref{Zulu}, for all time steps $\bar{T}_n=C\tilde{T}_n$ we have that
		\begin{equation*}
		\mathcal{G}(\rho(\bar{T}_n,\cdot))=	\mathcal{G}(\rho(C\tilde{T}_n,\cdot))=\mathcal{G}(\tilde{\rho}(\tilde{T}_n,\cdot))\geq  \tilde{T}_n^{-\frac{1}{s-1}}=\bar{T}_n^{-\frac{1}{s-1}}C^{\frac{1}{s-1}}\geq\bar{T}_n^{-\frac{1}{s-1}}
		\end{equation*}
		since $C>1$. Similarly 
		\begin{equation*}
		\|\rho(\bar{T}_n,\cdot)\|_{\dot{H}^{-1}}\geq  \bar{T}_n^{-\frac{2}{s-1}}\, .
		\end{equation*}

\section{Polynomial decay of certain cellular flows}
\label{grischun}
Using an ansatz of quasi self-similarity, the authors of \cite{AlbCrippa} construct explicit examples of a velocity fields with uniform in time bounds on the $\dot{W}^{s,p}$ norms (where $s>1$), which mix at a polynomial rate. Quasi self-similar velocity fields can be viewed as velocity fields of cellular type with a \textit{finite family} of basic building blocks. For completeness, as a simple generalization of the computations in Chapter 6 in \cite{AlbCrippa}, we show that any velocity field of cellular type with certain uniform bounds mix at polynomial rate. Some of those bounds entail regularity of order $\lceil s \rceil$, rather than $s$, where $\lceil s\rceil$ is the smallest integer greater or equal than $s$. This is due to the fact that Sobolev norms of integer order are local, which is a property used in the proof of Lemma 6.5 in \cite{AlbCrippa}.
\begin{theorem}
Let a divergence-free velocity field $u$ with corresponding solution $\rho$ of cellular type as in the Main Theorem be given. We denote by $\mathcal{A}\subset \mathcal{B}$ the set of all basic building blocks used to patch together $u$ and $\rho$, i.e. $(u_0,\rho_0)\in\mathcal{A}$ if and only if there exists an $n\in\mathbb{N}$ such that for some $Q\in\mathcal{T}_{\lambda^n}$ on $T_n\leq t\leq T_{n+1}$ we have
		\begin{equation*}
		u(t,x)=\frac{\lambda^n}{\tau^n}u_0\left(\frac{t-T_n}{\tau^n},\frac{x-r_Q}{\lambda^n}\right) \hspace{0.5cm}\textnormal{and}\hspace{0.5cm} \rho(t,x)=\rho_0\left(\frac{t-T_n}{\tau^n},\frac{x-r_Q}{\lambda^n}\right)\, .
		\end{equation*}
Under the assumption that for some $s>1$
\begin{equation}
\sup_{(u_0,\rho_0)\in\mathcal{A}}\sup_{0\leq r \leq 1}\|u_0(r,\cdot)\|_{\dot{W}^{\lceil s \rceil,p}(\mathbb{R}^2)}=C<\infty
\end{equation}
as well as the additional assumption that for a.e. $t\geq0$ we have that $u(t,\cdot)\in\dot{W}^{\lceil s \rceil,p}(\mathbb{R}^2)$, we can set  $\tau=\lambda^{1-s}$ which ensures that $u\in L^{\infty}(\dot{W}^{s,p})$ and that both the functional and the geometric mixing scale have the following \textit{polynomial} decay:
			\begin{equation}
			\mathcal{G}(\rho(t,\cdot))\leq C t^{-\frac{1}{s-1}}\hspace{0.5cm}\text{ and }\hspace{0.5cm}\|\rho(t,\cdot)\|_{\dot{H}^{-1}(\mathbb{R}^2)}\leq Ct^{-\frac{1}{s-1}}\, .
			\end{equation}

\end{theorem}
\begin{proof}
	The fact that under the assumptions of the Theorem choosing $\tau=\lambda^{1-s}$ will ensure that $u\in L^{\infty}(\dot{W}^{s,p})$ is a simple generalization of Lemma 6.5 in \cite{AlbCrippa}.\\\\
	As for the decay of the mixing norms, by the definition of a cellular flow and Lemma \ref{tilingmixing} we have that
		\begin{equation}
		\label{boomboogy}
		\mathcal{G}(\rho(t,\cdot))\leq C \lambda^n\hspace{0.5cm}\text{ and }\hspace{0.5cm}\|\rho(t,\cdot)\|_{\dot{H}^{-1}}\leq C \lambda^{n}
		\end{equation}
		for all $T_n\leq t\leq T_{n+1}$. For $t\in [T_n,T_{n+1})$ we have that
	\begin{equation*}
	t<T_{n+1}=\frac{\tau^{n+1}-1}{\tau-1}=\frac{\lambda^{-(s-1)(n+1)}-1}{\lambda^{1-s}-1}\, ,
	\end{equation*}
	which is equivalent to
	\begin{equation}
	\label{streetslon}
	\lambda^n< \left(\frac{1}{\lambda^{s-1}(t(\lambda^{1-s}-1)+1)}\right)^{\frac{1}{s-1}}\sim C(\lambda,s)t^{-\frac{1}{s-1}}\, . 
	\end{equation}
Therefore combining \eqref{boomboogy} and \eqref{streetslon} we obtain 
	\begin{equation}
	\mathcal{G}(\rho(t,\cdot))\leq Ct^{-\frac{1}{s-1}} \hspace{0.5cm}\text{ and }\hspace{0.5cm}\|\rho(t,\cdot)\|_{\dot{H}^{-1}}\leq Ct^{-\frac{1}{s-1}} \, . \qedhere
	\end{equation}
\end{proof}
	\section{Remarks on the universality of a mixer}
	\label{brent}
	\begin{definition}
	[Universal Mixer]
	We call a divergence-free velocity field $u$ a \textit{universal mixer}, if for any bounded, mean-free initial datum $\bar{\rho}$, for the the corresponding solution $\rho$ we have that
		\begin{equation}
		\mathcal{G}(\rho(t,\cdot))\to 0 \hspace{0.5cm}\text{as well as}\hspace{0.5cm}\|\rho(t,\cdot)\|_{\dot{H}^{-1}}\to 0
		\end{equation}
		as $t\to \infty$.
	\end{definition}
	The fact that a cellular flow cannot be a universal mixer can be seen by the following simple argument. For any cellular flow $u$, there exists a sub-square $Q\subset\subset \mathcal{Q}$ and $T>0$, such that $u(t,x)\cdot\nu(x)=0 $ on $\partial Q$ for all $t>T$, where $\nu(x)$ is the outer normal of $Q$. Thus if we choose $\rho_0$ such that $\rho(T,x)\equiv 1$ in $Q$, we would have $\rho(t,x)\equiv 1$ in $Q$ for all $t>T$, meaning that $u$ does not mix such $\rho_0$ as $t\to\infty$.\\\\
	In the following lemma we generalize this observation to a larger class of velocity fields. Note for instance that while the example in \cite{Lin} is not of cellular type, it still belongs to the larger class we are considering in Lemma \ref{dillaking}.
\begin{lemma}
	\label{dillaking}
	Let $u$ be a divergence-free velocity field with corresponding flow map $X$. For any $t>0$ we define
	\begin{equation}
	a(t):= \sup_{x\in\mathbb{R}^2}\left\lbrace\inf_{r}\left\lbrace X(s,x)\in B(X(t,x),r)\textnormal{ for all }s\geq t\right\rbrace\right\rbrace \in[0,+\infty]\,.
	\end{equation}
	If $a(t)\to 0$ as $t\to\infty$, then $u$ is not a universal mixer.
\end{lemma}
\begin{remark}
	Note that any velocity field of cellular type fulfills the condition of Lemma \ref{dillaking}, since for any $T_n\leq t <T_{n+1}$, we have that $a(t)\leq \sqrt{2}\lambda^n$.
\end{remark}
\begin{proof}
	Since $a(t)\to 0$, we can choose $t^*$ such that $|a(t)|\leq \frac{1}{100}$ for all $t\geq t^*$. We denote by $A$ the lower half of the unit square $\mathcal{S}=[0,1]^2$:
	\begin{equation}
	A=\left\lbrace (x_1,x_2)\in \mathbb{R}^2\, | \,0\leq x_1 \leq 1 \textnormal{ and }0\leq x_2 \leq \frac{1}{2} \right\rbrace\, .
	\end{equation}
	We implicitly define the initial data
	\begin{equation}
	\rho_0(x)=\begin{cases} 1 & \mbox{if } x\in X^{-1}(t^*,A)\\ -1 & \mbox{if }x\in X^{-1}(t^*,\mathcal{S}\setminus A)\\
	0 &\mbox{else }\end{cases}
	\end{equation}
	and note that $\rho_0$ is mean free, and (by definition) it is precisely the initial condition, for which the solution $\rho$ fulfills
	\begin{equation}
	\rho(t^*,x)=\begin{cases} 1 & \mbox{if } x\in A\\ -1 & \mbox{if }x\in \mathcal{S}\setminus A\\
	0 &\mbox{else. }\end{cases}
	\end{equation}
	Hence, if we choose $x$ in the center of $A$ (so $x=\left(\frac{1}{2},\frac{1}{4}\right)$), since $\dist(B(x,1/8), \mathbb{R}^2\setminus A)=1/8>1/100$ we note that
	\begin{equation}
	\rho\left(t,B(x,1/8)\right)=1
	\end{equation}
	for all $t\geq t^*$ and hence both mixing scales will not go to zero.
\end{proof}

\end{document}